\documentclass[12pt]{amsart}
\usepackage{amssymb,latexsym,amsthm}
\usepackage{graphicx}

\newcommand{\rad}{\mathrm{rad}}

\newcommand{\rr}{\mathrm{r}}

\newtheorem{theorem}{Theorem}[section]
\newtheorem{lemma}[theorem]{Lemma}
\newtheorem{cor}[theorem]{Corollary}

\theoremstyle{definition}

\newtheorem{example}[theorem]{Example}

\theoremstyle{remark}

\numberwithin{equation}{section}

\begin{document}

\baselineskip=17pt

\title[A property of an isometry]
{An algebraic property 
of an isometry between the groups of invertible elements 
in Banach algebras}

\author{Osamu~Hatori}
\address{Department of Mathematics, Faculty of Science, 
Niigata University, Niigata 950-2181 Japan}
\curraddr{}
\email{hatori@math.sc.niigata-u.ac.jp}

\thanks{The author was partly 
supported by the Grants-in-Aid for Scientific 
Research, The 
Ministry of Education, Science, Sports and Culture, Japan.}



\maketitle

\begin{abstract}
We show that if $T$ is an isometry (as metric spaces) between 
the invertible groups of 
unital Banach algebras, then 
$T$ is extended to a surjective real-linear isometry 
up to translation between the 
two Banach algebras. 
Furthermore if the underling algebras are 
closed unital 
standard operator algebras, $(T(e_A))^{-1}T$ is extended to 
a surjective real algebra isomorphism; 
if $T$ is a surjective isometry from the invertible 
group of a unital commutative Banach algebra onto that of 
a unital 
semisimple Banach algebra, then $(T(e_A))^{-1}T$ is extended to a 
surjective isometrical real algebra 
isomorphism between the two underling algebras.
\end{abstract}
%
%
%
%
%
%
%
%
\section{Introduction}
According to the definition the metric or the topological, 
and the algebraic structures of a Banach algebra are 
connected with each other. 
In the actual situation these structures are tightly connected
in the sense that 
some structure restores another one, for certain Banach algebras. 
The multiplication in a $C(X)$-space is restored by 
the structure as a Banach space;
the Banach-Stone theorem states that the existence of an isometric 
isomorphism as Banach spaces from the Banach algebra 
$C(X)$ 
of the complex valued continuous functions on a compact Hausdorff
space $X$ onto another one $C(Y)$ implies that 
$Y$ is homeomorphic to $X$, hence $C(X)$ is isometrically isomorphic as 
Banach algebras to $C(Y)$. Several generalizations 
including in \cite{ja1,ja2,ja3,naga} are investigated. 

Our main concern here is with the algebraic structure of isometries between
the invertible groups (the groups of all the invertible elements) 
of unital Banach algebras:
is an (metric-space) isometry 
between the invertible groups of unital 
(semisimple) Banach algebras 
multiplicative or anitimultiplicative, or preserving the square?  
Note that a unital surjective isometry between 
unital semisimple commutative Banach algebras need not be multiplicative even 
if the given isometry is assumed to be {\it complex-linear}.
We mainly
considered  commutative Banach algebras in \cite{hatori}.
In this paper we investigate with or without assuming being commutative, and
we show that a unital isometry from the invertible group 
in a closed unital standard operator algebras onto another one is 
multiplicative or antimultiplicative. 
We also show that a unital 
isometry from the invertible group of 
a unital commutative Banach algebra onto that 
of a unital semisimple Banach algebra is multiplicative. 
The hypothesis that the latter Banach algebra is semisimple is essential
(see Example \ref{dame}). 
\section{Extension of isometries}
In this section we show that an isometry between the groups of the invertible 
elements in unital Banach algebras is extended to an real-linear 
map up to  translation between the 
two Banach algebras of the 
form of a real-linear isometry followed by adding a radical element.

We begin by showing 
a local Mazur-Ulam theorem, which was proved in \cite{hatori}, 
with a proof for the sake of convenience.
\begin{lemma}\label{lmu}
Let ${\mathcal B}_1$ and ${\mathcal B}_2$ be real normed spaces, 
${\mathcal U}_1$ and ${\mathcal U}_2$ 
non-empty open subsets of ${\mathcal B}_1$ and 
${\mathcal B}_2$ respectively. Suppose that ${\mathcal T}$ is 
a surjective isometry from ${\mathcal U}_1$ onto 
${\mathcal U}_2$. 
If $f,g\in {\mathcal U}_1$ 
satisfy that $(1-r)f+rg\in {\mathcal U}_1$ for every 
$r$ with $0\le r \le 1$, then the equality
\[
{\mathcal T}(\frac{f+g}{2})=\frac{{\mathcal T}(f)+{\mathcal T}(g)}{2}
\]
holds.
\end{lemma}
\begin{proof}
Let $h,h'\in {\mathcal U}_1$. Suppose that $\varepsilon >0$ satisfies that 
$\frac{\|h-h'\|}{2}<\varepsilon$, and 
\[
\{u\in B_1:\|u-h\|<\varepsilon,\,\,\|u-h'\|<\varepsilon \}\subset {\mathcal U}_1,
\]
\[
\{a\in B_2:\|a-{\mathcal T}(h)\|<\varepsilon,\,\,
\|a-{\mathcal T}(h')\|<\varepsilon \}\subset {\mathcal U}_2.
\]
We will show that 
${\mathcal T}(\frac{h+h'}{2})=\frac{{\mathcal T}(h)+{\mathcal T}(h')}{2}$.
Set 
$r=\frac{\|h-h'\|}{2}$ and let
\[
L_1=
\{u\in {\mathcal B}_1:\|u-h\|=r=\|u-h'\|\},
\]
\[
L_2=
\{a\in {\mathcal B}_2:\|a-{\mathcal T}(h)\|=r=\|a-{\mathcal T}(h')\|\}.
\]
Set also $c_1=\frac{h+h'}{2}$ and $c_2=\frac{{\mathcal T}(h)+
{\mathcal T}(h')}{2}$. Then we have ${\mathcal T}(L_1)=L_2$, 
$c_1\in L_1 \subset {\mathcal U}_1$, and $c_2\in L_2 \subset {\mathcal U}_2$. 
Let 
\[
\psi_1(x)=h+h'-x \quad (x\in {\mathcal B}_1)
\]
and 
\[
\psi_2(y)={\mathcal T}(h)+{\mathcal T}(h')-y \quad (y\in {\mathcal B}_2).
\]
Then we see that $\psi_1(c_1)=c_1$, $\psi_1(L_1)=L_1$, 
and $\psi_2(L_2)=L_2$. 
Let $Q=\psi_1\circ{\mathcal T}^{-1}\circ\psi_2\circ{\mathcal T}$. 
A simple calculation shows that  
\[
2\|w-c_1\|=\|\psi_1(w)-w\|,\quad (w\in L_1)
\]
and 
\[
\|\psi_1(z)-w\|=\|\psi_1\circ Q^{-1}(z)-Q(w)\|,\quad (z,w \in L_1)
\]
hold. 
Applying these equations we see that
\begin{multline*}
\|Q^{2^{k+1}}(c_1)-c_1\|=\|\psi_1\circ Q^{2^{k+1}}(c_1)-c_1\|
\\
=\|\psi_1\circ Q^{2^k}(c_1)-Q^{2^k}(c_1)\|
=2\|Q^{2^k}(c_1)-c_1\|
\end{multline*}
hold for every nonzero integer $k$, where $Q^{2^n}$ denotes the 
$2^n$-time composition of $Q$. 
By induction we see for every non-negative integer $n$ that 
\[
\|Q^{2^n}(c_1)-c_1\|=2^{n+1}\|c_2-{\mathcal T}(c_1)\|
\]
holds. 
Since $Q(L_1)=L_1$ and $L_1$ is bounded we see that $c_2=
{\mathcal T}(c_1)$, i.e., ${\mathcal T}(\frac{h+h'}{2})=
\frac{{\mathcal T}(h)+{\mathcal T}(h')}{2}$. 

We assume that $f$ and $g$ are as described. Let
\[
K=\{(1-r)f+rg:0\le r\le 1\}.
\]
Since $K$ and ${\mathcal T}(K)$ are compact, there is $\varepsilon >0$ with
\[
d(K,{\mathcal B}_1\setminus {\mathcal U}_1)>\varepsilon, \quad 
d({\mathcal T}(K), {\mathcal B}_2\setminus {\mathcal U}_2)>\varepsilon,
\]
where 
$d(\cdot , \cdot )$ denotes the distance of two sets.
Then for every $h\in K$ we have 
\[
\{u\in {\mathcal B}_1:\|u-h\|<\varepsilon\}\subset {\mathcal U}_1
\]
and
\[
\{b\in {\mathcal B}_2:\|b-{\mathcal T}(h)\|<\varepsilon \}\subset {\mathcal U}_2.
\]
Choose a natural number $n$ with $\frac{\|f-g\|}{2^n}<\varepsilon$. 
Let 
\[
h_k=\frac{k}{2^n}(g-f)+f
\]
for each $0\le k\le 2^n$. By the first part of the proof we have
\[
{\mathcal T}(h_k)+{\mathcal T}(h_{k+2})-2{\mathcal T}(h_{k+1})=0
\qquad \text{($k$)}
\]
holds for $0\le k\le 2^n-2$. For $0\le k\le 2^n-4$,
adding the equations ($k$), 2 times of ($k+1$), and ($k+2$) we have
\[
{\mathcal T}(h_k)+{\mathcal T}(h_{k+4})-2{\mathcal T}(h_{k+2})=0,
\]
whence the equality 
\[
{\mathcal T}(\frac{f+g}{2})=\frac{{\mathcal T}(f)+{\mathcal T}(g)}{2}
\]
holds by induction on $n$.
\end{proof}
Note that an isometry between open sets of Banach algebras need not be 
extended to a linear isometry between these Banach algebras.
\begin{example}
Let $X=\{x,y\}$ be a compact Hausdorff space consisting of two points.
Let 
\[
{\mathcal U}=\{f\in C(X):\|f\|<1\}\cup
\{f\in C(X):\|f-f_0\|<1\},
\]
where $f_0\in C(X)$ is defined as $f_0(x)=0,\,f_0(y)=10$.
Suppose that 
\[
{\mathcal T}:{\mathcal U}\to {\mathcal U}
\]
is defined as ${\mathcal T}(f)=\tilde f$ if $\|f\|<1$ and 
${\mathcal T}(f)=f$ if $\|f-f_0\|<1$, where 
\begin{equation*}
\tilde f(t)=
\begin{cases}
-f(t),& t=x \\
f(t),& t=y.
\end{cases}
\end{equation*}
Then ${\mathcal T}$ is an isometry from ${\mathcal U}$ onto itself, 
while it cannot be extended to a real linear isometry up to translation.
\end{example}
Let $A$ be a unital Banach algebra. The group of all the invertible 
elements in $A$ is called the invertible group and is denoted 
by $A^{-1}$. The identity in $A$ is denoted by $e_A$.
The (Jacobson) radical for a given Banach algebra $A$ is denoted by 
$\rad (A)$. The spectrum of $a\in A$ is denoted by $\sigma (a)$ 
and $\rr (a)$ is the spectral radius for $a\in A$.

Surjective isometries between the invertible groups of 
unital Banach algebra
is extended to a real linear isometry up to translation (Theorem \ref{main}).
\begin{lemma}\label{rad}
Let $B$ be a unital Banach algebra and $a\in B$. Suppose that 
$\rr (fa)=0$ for every $f\in B^{-1}$. Then $a\in \rad (B)$.
\end{lemma}
\begin{proof}
First we will show that $\alpha a+e_B\in B^{-1}$ for every complex number 
$\alpha$. Suppose not. There is a complex number $\alpha_0$ with 
$\alpha_0 a +e_B \not\in B^{-1}$; $-1\in \sigma (\alpha_0 a)$. So 
$\alpha_0\ne 0$ and $-\frac{1}{\alpha_0}\in \sigma (a)=\sigma (e_Ba)$, 
hence $0<\rr (e_Ba)$, which contradicts to the assumption.

We will show that $a\in L$ whenever $L$ is a maximal left ideal of $B$, 
which will force that $a\in \rad (B)$. Suppose that there exists a left 
maximal ideal $L$ of $B$ with $a\not\in L$. Then $L+Ba$ is a left 
ideal of $B$ which properly contains $L$, so $L+Ba=B$ for 
$L$ is a maximal left ideal. Thus there is $f\in B$ with 
$fa+e_B\in L$. Let $\alpha$ be a complex number such that 
$f-\alpha e_B \in B^{-1}$; such an $\alpha$ exists since 
the spectrum is a compact set. 
Since $\alpha a+e_B\in B^{-1}$ by the first part of 
the proof, 
\begin{multline}
(\alpha a+e_B)^{-1}(f-\alpha e_B)a+e_B \\=
(\alpha a+e_B)^{-1}(f-\alpha e_B)a+
(\alpha a+e_B)^{-1}(\alpha a+e_B)\\=
(\alpha a+e_B)^{-1}(fa+e_B)\in L
\end{multline}
hold. Thus 
$(\alpha a+e_B)^{-1})(f-\alpha e_B)+e_B$ is singular, hence 
\[
\rr ((\alpha a+e_B)^{-1})(f-\alpha e_B)a)>0,
\] 
which is a contradiction since 
$(\alpha a+e_B)^{-1}(f-\alpha e_B)\in B^{-1}$.
\end{proof}

\begin{theorem}\label{main}
Let $A$ and $B$ be unital Banach algebras. Suppose that 
$T$ is a surjective isometry from $A^{-1}$ onto $B^{-1}$. Then there exists a 
surjective real-linear isometry $\tilde T_0$ from $A$ onto $B$ and 
$u_0\in \rad (B)$ such that 
$T(a)=\tilde T_0(a)+u_0$ for every $a\in A^{-1}$.
\end{theorem}
\begin{proof}
Since $T$ is an isometry, $\lim_{A^{-1}\ni a\to 0}T(a)$ exists. Let $u_0=
\lim_{A^{-1}\ni a\to 0}T(a)$. Let $f$ be an arbitrary element in $B^{-1}$. 
We will show that 
$\rr (fu_0)=0$, which will force that $u_0\in \rad (B)$ by Lemma \ref{rad}. 
Suppose that 
$\lambda \in \sigma (fu_0)$ and 
$\lambda \ne 0$. 
Let $c_{\lambda}=T^{-1}(-\lambda f^{-1})$. 
By Lemma \ref{lmu} 
\[
T(\frac{c_{\lambda}}{2})=T(\frac{(1-s)c_{\lambda}+sc_{\lambda}}{2})
=\frac{T((1-s)c_{\lambda})+T(sc_{\lambda})}{2}
\]
holds for every $0\le s\le 1$. Letting $s\to 0$, we see that 
\[
B^{-1}\ni T(\frac{c_{\lambda}}{2})=\frac{T(c_{\lambda})+u_0}{2}
=\frac{-\lambda f^{-1}+u_0}{2},
\]
so $-\lambda +fu_0 \in B^{-1}$, which is a contradiction 
since 
$\lambda \in \sigma (fu_0)$. 
Thus we see that $\sigma (fu_0)=\{0\}$, or $\rr (fu_0)=0$, so 
$u_0\in \rad (B)$ by Lemma \ref{rad}.

Define $T_0:A^{-1}\to B^{-1}$ by $T_0(a)=T(a)-u_0$. Since 
$u_0+B^{-1}=B^{-1}$ for $u_0 \in \rad(B)$ (cf. \cite[p.69]{dales}), $T_0$ is 
well-defined and bijective. 
We will show that $T_0(-f)=-T_0(f)$ for every $f\in A^{-1}$. 
Let $f\in A^{-1}$. Then $-f\in A^{-1}$, and for every integer $n$, 
$-f+\frac{i}{n}f\in A^{-1}$. 
We also see  
\[
(1-r)f+r(-f+\frac{i}{n}f) \in A^{-1}
\]
for every $0\le r \le 1$ and every  integer $n$. 
Then by Lemma \ref{lmu} 
\[
T_0\left( \frac{i}{2n}f\right)=T_0\left(\frac{f+(-f+\frac{i}{n}f)}{2}\right)
=\frac{T_0(f)+T_0(-f+\frac{i}{n}f)}{2}
\]
hold. Letting $n\to \infty$ we have 
\begin{equation}
T_0(-f)=-T_0(f).
\end{equation}

Next we will show that 
\begin{equation}\label{12}
T_0\left(\frac{f}{2}\right)=\frac{T_0(f)}{2}
\end{equation}
holds 
for every $f\in A^{-1}$. 
Let $f\in A^{-1}$. 
Then for every $1>\varepsilon >0$ and every $0\le r \le 1$
\[
(1-r)f +r\varepsilon f \in A^{-1}.
\]
Hence 
$T_0\left(\frac{f+\varepsilon f}{2}\right)=\frac{T_0(f)+T_0(\varepsilon f)}{2}$ 
holds by Lemma \ref{lmu}, then letting $\varepsilon \to 0$ 
the equation (\ref{12}) holds.

Let $f\in A^{-1}$. 
Suppose that $T_0(kf)=kT_0(f)$ holds for a positive integer $k$. 
Then 
\[
T_0\left( \frac{f+kf}{2}\right) =\frac{T_0(f)+T_0(kf)}{2}=\frac{(k+1)T_0(f)}{2}
\]
and by 
(\ref{12}) 
\[
T_0\left(\frac{f+kf}{2}\right)=\frac{T_0((k+1)f)}{2}
\]
holds, hence by induction 
$T_0(nf)=nT_0(f)$ holds for every positive integer $n$. 
Then for any pair of positive integers $m$ and $n$, 
\[
mT_0(\frac{n}{m}f)=T_0(m\frac{n}{m}f)=T_0(nf)=nT_0(f)
\]
holds, hence $T_0(\frac{n}{m}f)=\frac{n}{m}T_0(f)$ holds. 
By continuity of $T_0$, 
$T_0(rf)=rT_0(f)$ 
holds for every $f\in A^{-1}$ and $r>0$. 
Henceforth  
\begin{equation}\label{3.15}
T_0(rf)=rT_0(f)
\end{equation}
holds for every $f\in A^{-1}$ and for a non-zero real number $r$ 
since $T_0(-f)=-T_0(f)$.

Applying Lemma \ref{lmu} and (\ref{12}) we see that
\begin{equation}\label{sum}
T_0(f+g)=T_0(f)+T_0(g)
\end{equation}
holds for every pair $f$ and $g$ in $A^{-1}$ whenever 
$(1-r)f+rg\in A^{-1}$ holds for every $0\le r \le 1$. In particular 
(\ref{sum}) holds if $f,g\in \Omega_A$, where
\[
\Omega_A=\{a\in A:\text{$\|a-re_A\|<r$ holds for some positive real number 
$r$}\}
\]
is a convex subset of $A^{-1}$. 

Define the map $\tilde T_0:A\to B$ by $\tilde T_0(0)=0$ and 
\[
\tilde T_0(f)=T_0(f+2\|f\|e_A)-T_0(2\|f\|e_A)
\]
for a non-zero $f\in A$. The map $\tilde T_0$ is well-defined since 
$f+2\|f\|e_A$ and $2\|f\|e_A$ are 
in $\Omega_A$ for every non-zero $f\in A$ and 
$T_0$ is defined on $A^{-1}\supset \Omega_A$. 
If, in particular, $f\in \Omega_A$, then 
$T_0(f+2\|f\|e_A)=T_0(f)+T_0(2\|f\|e_A)$ holds, so that 
\begin{equation}\label{onomega}
\tilde T_0(f)=T_0(f)
\end{equation}
holds. 

We will show that $\tilde T_0$ is real-linear. Let $f\in A\setminus \{0\}$. Then 
$f+re_A\in \Omega_A$ for every $r\ge 2\|f\|$, whence by (\ref{sum})
\begin{multline*}
T_0(f+2\|f\|e_A)+T_0(re_A)\\
=T_0(f+2\|f\|e_A+re_A)=T_0(f+re_A)+T_0(2\|f\|e_A),
\end{multline*}
so that
\begin{equation}\label{UOmega}
\tilde T_0(f)=T_0(f+re_A)-T_0(re_A)
\end{equation}
holds for every $r\ge 2\|f\|$. Let $f,g \in A$. 
Then $\tilde T_0(f+g)=\tilde T_0(f)+\tilde T_0(g)$ holds if $f=0$ or $g=0$. Suppose that 
$f\ne 0$ and $g\ne 0$. Then by (\ref{sum}) and (\ref{UOmega}) we have
\begin{equation*}\begin{split}
\tilde T_0(f+g)&=
 T_0(f+g+2\|f\|e_A+2\|g\|e_A)-T_0(2\|f\|e_A+2\|g\|e_A) \\
&= T_0(f+2\|f\|e_A)+T_0(g+2\|g\|e_A)-T_0(2\|f\|e_A)-T_0(2\|g\|e_A) \\
&= \tilde T_0(f)+\tilde T_0(g)
\end{split}
\end{equation*}
holds. If $f=0$ or $r=0$ then $\tilde T_0(rf)=r\tilde T_0(f)$. Suppose that 
$f\ne 0$ and $r\ne 0$. If $r>0$, then by (\ref{3.15})
\begin{equation*}\begin{split}
\tilde T_0(rf) &=
T_0(rf+2\|rf\|e_A)-T_0(2\|rf\|e_A) \\
&= T_0(r(f+2\|f\|e_A))-T_0(r2\|f\|e_A) \\
&=rT_0(f+2\|f\|e_A)-rT_0(2\|f\|e_A)=r\tilde T_0(f)
\end{split}
\end{equation*}
If $r<0$, then
\[
\tilde T_0(rf)=(-r)\left(T_0(-f+2\|f\|e_A)-T_0(2\|f\|e_A)\right).
\]
Since $-f+2\|f\|e_A$, $f+2\|f\|e_A\in \Omega_A$ we have
\[
T_0(-f+2\|f\|e_A)-T_0(2\|f\|e_A)=-T_0(f+2\|f\|e_A)+T_0(2\|f\|e_A).
\]
It follows that 
\[
\tilde T_0(rf)=(-r)\left(-T_0(f+2\|f\|e_A)+T_0(2\|f\|e_A)\right)=r\tilde T_0(f).
\]

We will show that $\tilde T_0$ is surjective.
Let $a\in B$.  
Then 
\[
(T_0(e_A))^{-1}a+re_B\in \Omega_B\subset B^{-1},
\]
so 
\[
a+T_0(re_A)=a+rT_0(e_A)\in  B^{-1}
\]
holds whenever $\|(T_0(e_A))^{-1}a\|<r$ and $\|a\|<r$. We also have 
\[
\|T_0^{-1}(a+T_0(re_A))-r\|=\|a+T_0(re_A)-T_0(re_A)\|<r,
\]
thus $T_0^{-1}(a+T_0(re_A))\in \Omega_A$ holds 
whenever $\|(T_0(e_A))^{-1}a\|<r$ and $\|a\|<r$. Let $f= T_0^{-1}(a+T_0(re_A))-re_A
\in A$. 
Then 
$f+re_A=T_0^{-1}(a+T_0(re_A))\in \Omega_A$. Hence by (\ref{sum}) we see that 
\begin{multline*}
T_0(f+re_A)+T_0(2\|f\|e_A)\\
=T_0(f+2\|f\|e_A+re_A)=T_0(f+2\|f\|e_A)+T_0(re_A),
\end{multline*}
so we have 
\[
a=T_0(f+re_A)-T_0(re_A)=T_0(f+2\|f\|e_A)-T_0(2\|f\|e_A)=\tilde T_0(f).
\]

We will show that $\tilde T_0$ is an isometry. Since $\tilde T_0$ is linear, 
it is sufficient to show that 
$\|\tilde T_0(f)\|=\|f\|$ for every 
$f\in A$. If $f=0$, the equation clearly holds. Suppose 
that$f\ne 0$. Then 
\[
\|\tilde T_0(f)\|=\|T_0(f+2\|f\|e_A)-T_0(2\|f\|e_A)\|=\|f+2\|f\|e_A-2\|f\|e_A
\|=\|f\|
\]
hold.

We will show that $\tilde T_0$ is an extension of $T_0$, i.e., $\tilde T_0(f)=T_0(f)$ for 
every $f\in A^{-1}$. Let 
$P=\tilde T_0^{-1}\circ T_0:A^{-1}\to A$. 
For every $a\in A^{-1}$, 
\begin{equation}\label{1}
P(a+2\|a\|e_A)=a+2\|a\|e_A
\end{equation}
holds for $a+2\|a\|e_A \in \Omega_A$ and $T_0=\tilde T_0$ on $\Omega_A$ by 
(\ref{onomega}). Since 
$T_0(-f)=-T_0(f)$ holds for every $f\in A^{-1}$ and $\tilde T_0^{-1}$ is real-linear, 
we see that 
\begin{equation}\label{2}
P(a-2\|a\|e_A)=-P((-a)+2\|-a\|e_A)=a-2\|a\|e_A
\end{equation}
holds for every $a\in A^{-1}$.

We will show that 
\[
P(a\pm 2i\|a\|e_A)=a\pm 2i\|a\|e_A
\]
holds for every $a\in A^{-1}$. 
Since 
\begin{multline*}
\|t(a+2\|a\|e_A)+(1-t)(\pm 2i\|a\|e_A)
-2(t\pm (1-t)i)\|a\|e_A\| \\
=t\|a\|<2|t\pm (1-t)i|\|a\|
\end{multline*}
hold
\[
t(a+2\|a\|e_A)+(1-t)(\pm 2i\|a\|e_A)\in A^{-1}
\]
holds for every $0\le t\le 1$ we see by (\ref{sum}) that 
\[
T_0(a+2\|a\|e_A)+T_0(\pm 2i\|a\|e_A)=T_0(a+
2\|a\|e_A\pm 2i\|a\|e_A).
\]
In a way similar we have
\[
T_0(a\pm 2i\|a\|e_A)+T_0(2\|a\|e_A)=
T_0(a\pm 2i\|a\|e_A+2\|a\|e_A),
\]
hence
\begin{multline}\label{star}
\tilde T_0 (a)= T_0(a+2\|a\|e_A)-T_0(2\|a\|e_A) \\
=T_0(a\pm 2i\|a\|e_A)-T_0(\pm 2i\|a\|e_A)
\end{multline}
holds for every $a\in A^{-1}$. For every $0\le t\le 1$
\[
t(\pm 2i\|a\|e_A)+(1-t)4\|a\|e_A \in A^{-1}
\]
holds, hence
\[
T_0(\pm 2i\|a\|e_A+4\|a\|e_A)=
T_0(\pm 2i\|a\|e_A)+T_0(4\|a\|e_A),
\]
so that 
\[
T_0(\pm 2i\|a\|e_A)=
T_0(\pm 2i\|a\|e_A+4\|a\|e_A)-T_0(4\|a\|e_A)=
\tilde T_0(\pm 2i\|a\|e_A).
\]
Thus we see by (\ref{star}) that 
\begin{multline}\label{3}
P(a\pm 2i\|a\|e_A)=
\tilde T_0^{-1}(T_0(a\pm 2i\|a\|e_A)) \\
=\tilde T_0^{-1}(\tilde T_0(a)+T_0(\pm 2i\|a\|e_A))\\
=\tilde T_0^{-1}(\tilde T_0(a)+\tilde T_0(\pm 2i\|a\|e_A))=
a\pm 2i \|a\|e_A
\end{multline}
holds for every $a\in A^{-1}$ since $\tilde T_0$ is real-linear.
Applying (\ref{1}) and (\ref{2})
\begin{multline}\label{5}
2\|a\|=\|a\pm 2\|a\|e_A-a\|=\|P(a\pm 2\|a\|e_A)-P(a)\| \\
\|a\pm 2\|a\|e_A -P(a)\|=
\|P(a)-a\pm 2\|a\|e_A\|
\end{multline}
holds for every $a\in A^{-1}$. 
In a same way we have by (\ref{3}) that 
\begin{equation}\label{55}
2\|a\|=\|P(a)-a\pm 2i\|a\|e_A\|
\end{equation}
holds for every $a\in A^{-1}$. 
For an element $b\in B$ the numerical range of $b$ is denoted by 
$W(b)$. 
By (\ref{55}) and \cite[Lemma 2.6.3]{palmer} 
\begin{multline}
\sup \{\mathrm{Im}(\lambda):\lambda \in W(P(a)-a)\} \\
=\inf_{t>0}t^{-1}(\|e_A-it(P(a)-a)\|-1) \\
\le 
2\|a\|(\|e_A-\frac{i}{2\|a\|}(P(a)-a)\|-1)=0.
\end{multline}
Since $W(-P(a)+a)=-W(P(a)-a)$ we have
\begin{multline}
-\inf \{\mathrm{Im} (\lambda):\lambda \in W(P(a)-a)\} \\
=\sup \{\mathrm{Im}(\lambda): \lambda \in W(-P(a)+a)\} \\
=
\inf _{t>0} t^{-1}(\|e_A-it(-P(a)+a)\|-1) \\
\le
2\|a\|(\|e_A+\frac{i}{2\|a\|}(P(a)-a)\|-1)=0
\end{multline}
Thus we see that 
\[
W(P(a)-a)\subset \mathbb{R},
\]
where $\mathbb{R}$ denotes the set of real numbers. 
Applying (\ref{5}) and \cite[Lemma 2.6.3]{palmer} 
in a same way we see that 
\[
iW(P(a)-a)=W(i(P(a)-a))\subset \mathbb{R}.
\]
It follows that 
\[
W(P(a)-a)=\{0\}.
\]
Since 
\[
\|P(a)-a\|\le e\|P(a)-a\|_W
\]
holds (
cf. \cite[Theorem 2.6.4]{palmer}), where $\|\cdot \|_W$ denotes the 
numerical radius, 
we see that $P(a)=a$ holds for every $a\in A^{-1}$.

Since $T(a)=T_0(a)+u_0$ for $a\in A^{-1}$ by the definition of $T_0$, 
we conclude that $T(a)=\tilde T_0(a)+u_0$ holds for every $a\in A^{-1}$.
\end{proof}
\section{Multiplicativity or antimultiplicativity of isometries}
We proved the following in \cite{hatori}. The proof involves much about 
commutativity and  semisimplicity of the given Banach algebra $A$.
\begin{theorem}\label{st}
Let $A$ be a unital semisimple commutative Banach algebra and $B$ 
a unital Banach algebra. Suppose ${\mathfrak A}$ and ${\mathfrak B}$
are open subgroups of $A^{-1}$ and $B^{-1}$ respectively. 
Suppose that $T$ is a surjective isometry 
from ${\mathfrak A}$ onto
${\mathfrak B}$. Then $B$ is a semisimple and commutative, and 
$(T(e_A))^{-1}T$ is extended to an isometrical real algebra 
isomorphism from $A$ onto $B$. In particular, 
$A^{-1}$ is isometrically isomorphic to 
$B^{-1}$ as a metrizable group.
\end{theorem}
In the following comparison result as the above we make use of Theorem 
\ref{main}.
\begin{cor}\label{comsem}
Let $A$ be a unital commutative Banach algebra and 
$B$ a semisimple Banach algebra.
Suppose that $T$ is a surjective isometry from $A^{-1}$ onto 
$B^{-1}$. Then $(T(e_A))^{-1}T$ is extended to a surjective isometrical 
real algebra isomorphism from $A$ onto $B$. Moreover, 
$A$ is semisimple and $B$ is commutative. In particular, 
$A^{-1}$ is isometrically isomorphic to $B^{-1}$ as a metrizable group.
\end{cor}
\begin{proof}
By Theorem \ref{main} there is a $u_0\in \rad (B)$ such that $T-u_0$ is 
extended to a real-linear isometry from $A$ onto $B$. Since 
$B$ is semisimple (cf. \cite[Theorem 2.5.8]{dales}), $u_0=0$, hence 
$T$ is extended to a surjective real-linear 
isometry $\tilde T$ from $A$ onto $B$ since $B$ is semisimple. We will 
show that $A$ is semisimple. Let $a\in \rad (A)$ and let
$T_a:A^{-1}\to B^{-1}$ be defined as $T_a(b)=T(a+b)$ for $b\in A^{-1}$. 
Then $T_a$ is well-defined and a surjective isometry since $a+A^{-1}=A^{-1}$ 
for $a\in \rad (A)$. By Theorem \ref{main} 
$T_a$ is also extended to a surjective real-linear isometry $\tilde T_a$ from 
$A$ onto $B$. For every positive integer
\[
\tilde T_a\left(\frac{e_A}{n}\right)=T_a\left(\frac{e_A}{n}\right)=T
\left(a+\frac{e_A}{n}\right)=\tilde T\left(a+\frac{e_A}{n}\right)
\]
holds. Letting $n\to \infty$ we have
\[
0=\tilde T(a),
\]
hence $a=0$ for $\tilde T$ is injective. It hollows that $\rad (A)=\{0\}$, 
or $A$ is semisimple. 
Then by Theorem \ref{st} the conclusion holds.
\end{proof}

The hypothesis that $B$ is semisimple in Corollary \ref{comsem} is essential 
as the following example (cf. \cite{hatori}) shows that a unital isometry 
from $A^{-1}$ onto $B^{-1}$ need not be multiplicative nor antimultiplicative 
unless at least one of $A$ or $B$ are  semisimple.
\begin{example}\label{dame}
Let 
\[
A_0=\{
\left(
\begin{smallmatrix}
0&a&b \\0&0&c \\ 0&0&0
\end{smallmatrix}
\right)
:a,\,\,b,\,\,c\in {\mathbb C}\}.
\]
Let 
\[
A=\{
\left(
\begin{smallmatrix}
\alpha&a&b \\0&\alpha&c \\ 0&0&\alpha
\end{smallmatrix}
\right)
:\alpha,\,\,a,\,\,b,\,\,c\in {\mathbb C}\}
\]
be the unitization of $A_0$, where the multiplication (in $A_0$) is the 
zero multiplication; $MN=0$ for every $M,N\in A_0$.
Let $B=A$ as sets, while the multiplication in $B$ is the usual multiplication 
for matrices. 
Then $A$ and $B$ are unital Banach algebras under the usual operator norm. 
Note that $A$ is commutative and $A$ nor $B$ are not semisimple. 
Note also that 
$A^{-1}=\{\left(
\begin{smallmatrix}
\alpha&a&b \\0&\alpha&c \\ 0&0&\alpha
\end{smallmatrix}
\right)\in A:\alpha\ne 0\}$ and $B^{-1}=\{\left(
\begin{smallmatrix}
\alpha&a&b \\0&\alpha&c \\ 0&0&\alpha
\end{smallmatrix}
\right)\in B:\alpha\ne 0\}$.
Define 
$T:A^{-1}\to B^{-1}$ by $T(M)=M$. 
Then $T$ is well-defined and a surjective isometry.
On the other hand $A^{-1}$ is not (group) isomorphic to $B^{-1}$, 
in particular, $T$ is not multiplicative nor antimultiplicative.
\end{example}
We show a positive result for standard operator algebras.
\begin{cor}\label{standard}
Let $X$ (resp. $Y$) be a Banach space. Suppose that $A$ (resp. $B$) 
is a unital closed 
subalgebra of ${\mathfrak B}(X)$ (resp. ${\mathfrak B}(Y)$), 
the Banach algebra of all the bounded operators on $X$ (resp. $Y$), 
which contains all finite rank operators. 
Suppose that $T$ is a surjective isometry from 
$A^{-1}$ onto $B^{-1}$. Then there exists an invertible bounded linear 
or conjugate linear operator $U:X\to Y$ such that $T(a)=T(e_A)UaU^{-1}$ 
for every 
$a\in A^{-1}$, or there exists an invertible bounded linear or conjugate 
linear operator $V:X^*\to Y$ such that $T(a)=T(e_A)Va^*V^{-1}$ for every $a\in A^{-1}$. In particular, if $T$ is unital in the sense that 
$T(e_A)=e_B$, then $T$ is multiplicative or antimultiplicative.
\end{cor}
\begin{proof}
By Theorem \ref{main} there is a $u_0\in \rad (B)$ such that $T-u_0$ is 
extended to a real-linear isometry $\tilde T_0$ from $A$ onto $B$. Since 
$B$ is semisimple (cf. \cite[Theorem 2.5.8]{dales}), $u_0=0$, hence 
$\tilde T_0=T$ on $A^{-1}$. Thus $(T(e_A))^{-1}\tilde T_0$ 
additive surjection such that $(T(e_A))^{-1}\tilde T_0(A^{-1})
=B^{-1}$. Applying  Theorem 3.2 in \cite{houcui} for 
$(T(e_A))^{-1}\tilde T_0$,
there exists an invertible bounded linear or conjugate linear operator 
$U:X\to Y$ such that $(T(e_A))^{-1}\tilde T_0(a)=UaU^{-1}$ ($a\in A$), 
or there exists an invertible bounded linear or conjugate linear 
operator 
$V:X^*\to Y$ such that $(T(e_A))^{-1}\tilde T_0(a)=Va^*V^{-1}$ ($a\in A$). 
Henceforth the conclusion holds.
\end{proof}
Let $M_n$ be the algebra of all $n\times n$ matrices over the complex number 
field. For $M\in M_n$ the spectrum  is denoted by $\sigma (M)$ and $M^t$ 
is the transpose of $M$. $E$ denotes the identity matrix. 
Let $\|\cdot\|$ and $\|\cdot\|'$ 
denote any matrix norms on $M_n$ (cf. \cite{hojo}).
\begin{cor}
If $S$ is a surjection from the group $M_n^{-1}$ of the invertible $n\times n$ 
matrices over the complex number field onto itself such that 
$\|S({M})-S({N})\|'=\|{M}-{N}\|$ for all ${M},{N}\in M_n^{-1}$, then 
there exists an invertible matrix ${U}\in M_n$ such that 
$S({M})=S(E){U}{M}{U}^{-1}$ for all ${M}\in M_n^{-1}$, or
$S({M})=S(E){U}{M}^t{U}^{-1}$ for all ${M}\in M_n^{-1}$, or 
$S(M)=S(E)U\overline{M}U^{-1}$  for all $M\in M_n^{-1}$, or
$S(M)=S(E)U\overline{M}^tU^{-1}$  for all $M\in M_n^{-1} hold$.
In particular, if $S$ is unital, then $S$ is multiplicative or 
antimultiplicative.
\end{cor}
\begin{proof}
By Corollary \ref{standard} there is an invertible matrix $U$ such that 
one of 
the following four occurs.
\begin{enumerate}
\item
$S(M)=S(E)UMU^{-1}$ holds for every $M\in M_n^{-1}$, 
\item
$S(M)=S(E)U\overline{M}U^{-1}$ holds for every $M\in M_n^{-1}$, 
\item
$S(M)=S(E)UM^tU^{-1}$ holds for every $M\in M_n^{-1}$, 
\item
$S(M)=S(E)U\overline{M}^tU^{-1}$ holds for every $M\in M_n^{-1}$. 
\end{enumerate}
Henceforth the conclusion holds.
\end{proof}


\begin{thebibliography}{99}





\bibitem{dales}
H.~G.~Dales,
"Banach Algebras and Automatic Continuity", 
London Mathematical Society Monographs. New Series, 24. Oxford Science Publications. The Clarendon Press, Oxford University Press, New York, 2000




\bibitem{hatori}
O.~Hatori,
{\it Isometries between groups of invertible elements in Banach algebras},
to appear


\bibitem{hojo}
R.~A.~Horn and C.~R.~Johnson,
"Matrix analysis
 Corrected reprint of the 1985 original", Cambridge University Press, Cambridge, 1990

\bibitem{houcui}
J.~Hou and J.~Cui,
{\it Additive maps on standard operator algebras preserving invertibilities or zero divisors},
Linear Algebra Appl.,
{\bf 359}(2003), 219--233




\bibitem{ja1}
K.~Jarosz,
{\it The uniqueness of multiplication in function algebras},
Proc. Amer. Math. Soc.,
{\bf 89}(1983), 249--253



\bibitem{ja2}
K.~Jarosz,
{\it Isometries in semisimple, commutative Banach algebras},
Proc. Amer. Math. Soc.,
{\bf 94}(1985), 65--71



\bibitem{ja3}
K.~Jarosz,
"Perturbations of Banach algebras",
Lecture Notes in Mathematics, vol.1120 Springer-Verlag, Berlin, 1985








\bibitem{naga}
M.~Nagasawa,
{\it Isomorphisms between commutative Banach algebras with an 
application to rings of analytic functions},
K\={o}dai Math. Sem. Rep.,
{\bf 11}(1959), 182--188


\bibitem{palmer}
T.~W.~Palmer,
"Banach algebras and the general theory of $\sp *$-algebras. Vol. I. Algebras and Banach algebras",
Encyclopedia of Mathematics and its Applications, 49. Cambridge University Press, Cambridge, 1994.






\end{thebibliography}
\end{document}